\documentclass{article}

\usepackage[utf8]{inputenc}
\usepackage[english]{babel}
\usepackage{amsthm}
\usepackage{amsmath}
\usepackage{graphicx}
\usepackage{float}
\usepackage{tikz} 
\usepackage{url} 

\usepackage[colorlinks=true]{hyperref}%

\providecommand{\keywords}[1]{\textbf{\textit{Keywords---}} #1}
\newcommand{\stargraph}[2]{\begin{tikzpicture}
    \node[circle,fill=black] at (360:0mm) (center) {};
    \foreach \n in {1,...,#1}{
        \node[circle,fill=black] at ({\n*360/#1}:#2cm) (n\n) {};
        \draw (center)--(n\n);
    }
\end{tikzpicture}}

\newtheorem{theorem}{Theorem}[section]
\newtheorem{conjecture}[theorem]{Conjecture}
\newtheorem{corollary}[theorem]{Corollary}
\newtheorem{definition}[theorem]{Definition}
\newtheorem{lemma}[theorem]{Lemma}
\newtheorem{observation}[theorem]{Observation}
\newtheorem{problem}[theorem]{Problem}

\title{Inequalities Connecting the Annihilation and Independence Numbers}
\author{
        Ohr Kadrawi \\
        Department of Mathematics\\
        Ariel University\\
        Ariel 4070000, \underline{Israel}\\
        orka@ariel.ac.il
        \and
        Vadim E. Levit \\
        Department of Mathematics\\
        Ariel University\\
        Ariel 4070000, \underline{Israel}\\
        levitv@ariel.ac.il
}

\begin{document}

    \maketitle
       
    \begin{abstract}
    
        Given a graph $G$, the number of its vertices is represented by $n(G)$, while the number of its edges is denoted as $m(G)$.
        An \textit{independent set} in a graph is a set of vertices where no two vertices are adjacent to each other and the size of the maximum independent set is denoted by $\alpha(G)$. A \textit{matching} in a graph refers to a set of edges where no two edges share a common vertex and the maximum matching size is denoted by $\mu(G)$. If $\alpha(G) + \mu(G) = n(G)$, then the graph $G$ is called a \textit{K\"{o}nig--Egerv\'{a}ry graph}.
        
        Considering a graph $G$ with a degree sequence $d_1 \leq d_2 \leq \cdots \leq d_n$, the \textit{annihilation number} $a(G)$ is defined as the largest integer $k$ such that the sum of the first $k$ degrees in the sequence is less than or equal to $m(G)$ (Pepper, 2004).
        
        It is a known fact that $\alpha(G)$ is less than or equal to $a(G)$ for any graph $G$. Our goal is to estimate the difference between these two parameters. Specifically, we prove a series of inequalities, including
        $a(G) - \alpha(G) \leq \frac{\mu(G) - 1}{2}$ for trees, $a(G) - \alpha(G) \leq 2 + \mu(G) - 2\sqrt{1 + \mu(G)}$ for bipartite graphs and $a(G) - \alpha(G) \leq \mu(G) - 2$ for K\"{o}nig--Egerv\'{a}ry graphs. Furthermore, we demonstrate that these inequalities serve as tight upper bounds for the difference between the annihilation and independence numbers, regardless of the assigned value for $\mu(G)$.
        
    \end{abstract}

    \keywords{annihilation number, independence number, tree, bipartite graph, K\"{o}nig--Egerv\'{a}ry graph}.

    \newpage
    \section{Introduction}

        In this paper, we consider a finite, undirected graph $G = (V, E)$ without loops or multiple edges. The graph has a vertex set denoted by $V(G)$ with a cardinality of $|V(G)| = n(G)$, and an edge set denoted by $E(G)$ with a cardinality of $|E(G)| = m(G)$.
        
        A subset $S \subseteq V(G)$ is considered \textit{independent} if no two vertices in $S$ are adjacent. The collection of all independent sets of $G$ is denoted as $Ind(G)$. A \textit{maximum independent set} of $G$ is an independent set with the largest possible size. The \textit{independence number} of $G$ is denoted as $\alpha(G)$ and is defined as the maximum cardinality among all sets $S \in Ind(G)$.

        A \textit{matching} in a graph $G$ refers to a set of edges $M \subseteq E(G)$ where no two edges in $M$ share a common vertex. A \textit{maximum matching}, denoted as $\mu(G)$, is a matching with the largest possible cardinality. The cardinality of a maximum matching is called the \textit{matching number} of the graph.

        A graph is a \textit{bipartite} if and only if it does not contain odd cycles. Clearly, every subgraph of a bipartite graph is also bipartite.
        
        For any graph $G$, it is well-known, as stated in  \cite{BorosGolumbicLevit2002}, that the following inequalities hold: $\alpha(G) + \mu(G) \leq n(G) \leq \alpha(G) + 2\mu(G)$. If a graph $G$ satisfies the condition $\alpha(G) + \mu(G) = n(G)$, it is referred to as a \textit{K\"{o}nig--Egerv\'{a}ry} graph, as mentioned in \cite{Deming1979, Sterboul1979}. Notably, all bipartite graphs and trees belong to the class of K\"{o}nig--Egerv\'{a}ry graphs.

        Consider the \textit{degree sequence} of a graph $G$ given by $d_1 \leq d_2 \leq \ldots \leq d_a \leq d_{a+1} \leq \ldots \leq d_n$. The \textit{annihilation number} of $G$, denoted as $a(G)$, was introduced by Pepper \cite{Pepper2009, Pepper2004}. It is defined as the largest integer $k$ such that the sum of the first $k$ terms in the degree sequence is no greater than half the sum of all the degrees in that sequence. In other words, $a(G)$ is the maximum value of $k$ satisfying $\sum_{i=1}^{k} d_i \leq m(G)$.
        
        Let $A \subseteq V(G)$ be a subset of vertices. The notation $deg(A)$ denotes the sum of the degrees of the vertices in $A$, i.e., $\sum_{v \in A}^{} d(v)$. An \textit{annihilating set} refers to any subset $A \subseteq V(G)$ that satisfies the condition $deg(A) \leq m(G)$. It is evident that every independent set is also an annihilating set.

        An annihilating set $A$ is considered \textit{maximal} if for every vertex $x \in V(G) - A$, adding $x$ to $A$ results in $deg(A \cup {x}) > m(G)$. On the other hand, an annihilating set $A$ is called \textit{maximum} if its cardinality is equal to the annihilation number of the graph, i.e., $|A| = a(G)$, as stated in \cite{Pepper2004}.

        For instance, consider a graph $G = K_{p,q} = (A, B, E)$ where $p > q$. In this case, $A$ is a maximum annihilating set, while $B$ is a maximal annihilating set.

        \begin{theorem} \cite{LevitMandrescu2020, Pepper2004}
            For every graph $G$, $a(G) \geq max\{\alpha(G), \frac{n(G)}{2}\}$.
        \end{theorem}
        Extensive research has been conducted to explore the relationship between the annihilation number and various parameters of a graph 
        \cite{  Amjadi2015, AramKhoeilarSheikholeslamiVolkmann2018, BujtasJakovac, DehgardiNorouzianSheikholeslami2013, DehgardiSheikholeslamiKhodkar2013, DehgardiSheikholeslamiKhodkar2014, DesormeauxHaynesHenning2013, GentnerHenningRautenbach, HuaXuHua2023, Jakovac2019, JaumeaMolina2018, LarsonPepper2011, LevitMandrescu2020, LevitMandrescu2022, NingLuWang2019, Pepper2004}.

        \begin{lemma}
            For every graph $G$, $a(G) - \alpha(G) \leq n(G) - 1$.
        \end{lemma}
        \begin{proof} The annihilation number $a(G)$ is bounded by the number of vertices $n(G)$ since $a(G)$ corresponds to an index in the degree sequence. The equality $a(G) = n(G)$ can only occur when the graph has no edges, i.e., $m(G) = 0$. It is worth noting that every single vertex in a graph forms an independent set, so for any graph $G$, we have $\alpha(G) \geq 1$. Therefore, this lemma holds for every graph.
        \end{proof} 
        
        \begin{lemma}\label{a_minum_alpha_leq_mu}
            If $G$ is a K\"{o}nig--Egerv\'{a}ry graph, then $0 \leq a(G) - \alpha(G) \leq \mu(G)$. Moreover, if $m(G)>0$, then $a(G) - \alpha(G) < \mu(G)$.
        \end{lemma}    
         \begin{proof} Pepper \cite{Pepper2004} demonstrated that the inequality $0 \leq a(G) - \alpha(G)$ holds. As for the right side, we know that $a(G) \leq n(G) = \alpha(G) + \mu(G)$. Consequently, we have $a(G) - \alpha(G) \leq \mu(G)$.
         
         If $m(G)>0$, then $a(G) < n(G)$, which completes the proof.
        \end{proof} 
        \begin{lemma}\label{KE_alpha>=mu}
            For every K\"{o}nig--Egerv\'{a}ry graph $G$, $\mu(G) \leq \alpha(G)$.
        \end{lemma}
        \begin{proof}
        Notice that $\mu(G)$ is less than or equal to half the value of $n(G)$ because every edge $e \in M(G)$ consists of two vertices and no two edges in $M(G)$ share common vertices. Additionally, $\alpha(G) = n(G)-\mu(G)$, so we can conclude that:
        
        \[\mu(G) \leq \frac{n(G)}{2} \leq \alpha(G).\]
        \end{proof}
        In this paper, we focus on determining tight upper bounds for the difference between the annihilation number and the independence number for different types of graphs. Specifically, in Section \ref{tree_section}, we examine trees, in Section \ref{bipartite_section}, we analyze bipartite graphs, and in Section \ref{KE_section}, we investigate K\"{o}nig--Egerv\'{a}ry graphs.
        

    \section{The tight upper bound for trees}\label{tree_section}
        The initial inequality concerning the difference between the annihilation number and the independence number of trees can be stated as follows:
        \[a(T)-\alpha(T)\leq \mu(T).\] This inequality is derived from Lemma \ref{a_minum_alpha_leq_mu}, considering that a tree is a K\"{o}nig--Egerv\'{a}ry graph.\\
        
        To establish a tighter bound, we define the annihilation decomposition.

        \begin{definition}\label{definition_annihilation_decomposition}
            An \textit{annihilation decomposition} of a graph $G$ is a partition $\langle A,B \rangle$ of its vertex set to a maximum annihilation set and its complement. In what follows, we define $k\langle A,B \rangle$ as the number of edges between vertices of $A$ and $B$.
        
            \begin{figure}[H]
                \centering
                    \begin{tikzpicture}
                    \draw[] (0,0) ellipse (0.5cm and 1cm);
                    \draw[] (4,0) ellipse (0.5cm and 1cm);
                    \node at (0,0) {A};
                    \node at (4,0) {B};
                    \node at (2,0.5) {$k\langle A,B \rangle$};
                    \draw[thick] (0.7,0) -- (3.3,0) ; 
                    \draw[thick] (0.7,0.3) -- (3.3,0.3) ;
                    \draw[thick] (0.7,-0.3) -- (3.3,-0.3) ;
                \end{tikzpicture}
                \caption{An annihilation decomposition of a graph.}
                \label{fig:decomposition}
            \end{figure}
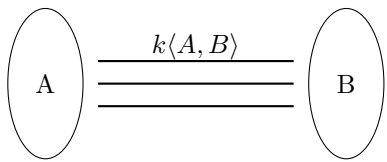
        \end{definition}
        It is important to note that in an annihilation decomposition, $|A| = a(G)$ and $|B| = n(G)-a(G)$.

        \begin{lemma}\label{lemma:ma<mb}
            In an annihilation decomposition $\langle A,B \rangle$ of a graph $G$, it holds that $m(G[A]) \leq m(G[B])$.
        \end{lemma}
        
        \begin{proof}
            It is known that the sum of all degrees in a graph is equal to twice the number of edges, i.e., $2m(G)$. According to the definition of the annihilation number, $a$ represents the maximum index in the degree sequence such that the sum of the degrees of vertices $v_i$ for $i \leq a$ is less than or equal to $m(G)$. These vertices are included in the set $A$.
            
            On the other hand, the sum of the degrees of vertices in set $B$ is greater than or equal to $m(G)$. Consequently, there are more (or an equal number of) edges in $G[B]$ compared to $G[A]$.
        \end{proof}

        \begin{lemma}\label{lemma:in_bipartite_a-alpha<m(A)}
            Let $G$ be a bipartite graph with an annihilation decomposition $\langle A,B \rangle$. Then $a(G) - \alpha(G) \leq m(G[A])$.
        \end{lemma}

        \begin{proof}
            First, notice that for any graph $G$, the maximum matching number $\mu(G)$ is always less than or equal to the number of edges $m(G)$ because $M \subseteq E(G)$.\\

            Let us now prove the inequality $a(G) - \alpha(G) \leq m(G[A])$ by contradiction.\\

            Suppose that $m(G[A]) < a(G) - \alpha(G)$. Then, 
            \[
             a(G) - \alpha(G[A])=n(G[A]) - \alpha(G[A])=\mu(G[A]) \leq m(G[A])< a(G) - \alpha(G).
            \]
            This implies that $\alpha(G[A]) > \alpha(G)$, which leads to a contradiction. Hence, the assumption that $m(G[A]) < a(G) - \alpha(G)$ must be false, and we conclude that $a(G) - \alpha(G) \leq m(G[A])$.
        \end{proof}
        
        In the next part of this section, our objective is to establish the tight upper bound specifically for trees. To accomplish this, we can utilize Lemma \ref{lemma:ma<mb} and Lemma \ref{lemma:in_bipartite_a-alpha<m(A)}, adapting them for trees by substituting $T$ in place of $G$.\\
        
        Let $T$ be a tree with annihilation decomposition $\langle A, B \rangle$. The sum of the degrees in $A$ is equal to twice the number of edges in $A$ (since each edge has two endpoints in $A$), which is denoted as $2m(T[A])$ plus the edges between $A$ and $B$ (counted once), which is equal to $k\langle A,B \rangle$. Using Lemma \ref{lemma:ma<mb}, we can obtain the following inequality:
         
        \[\color{blue}2m(T[A]) + k\langle A,B \rangle \leq \color{black}m(T) \leq \color{brown}2m(T[B]) + k\langle A,B \rangle \color{black}.\]
        
        By employing Lemma \ref{lemma:in_bipartite_a-alpha<m(A)}, we can establish the inequality $a(T) - \alpha(T) \leq m(T[A])$ on the left side of the aforementioned inequality. Consequently, we can express it as:
        \[\color{blue}2(a(T) - \alpha(T)) + k\langle A,B \rangle \leq 2(m(T[A])) + k\langle A,B \rangle \color{black}.\]
        
        On the right side, since $T[B]$ is a subgraph of a tree, it can be regarded as a forest, and the maximum number of edges it can have is $n(T[B])-1$. Thus, we can write:
        \[\color{brown}2m(T[B]) + k\langle A,B \rangle \leq 2(n(T[B])-1) + k\langle A,B \rangle \color{black}.\]
        
        By combining the above inequalities, we obtain the following:
        
        \begin{theorem}
            For every tree $T$, $a(T) - \alpha(T) \leq \frac{\mu(T)}{2} - \frac{1}{2}$.
        \end{theorem}
            
        \begin{proof}
            Based on the aforementioned calculations, we can represent these inequalities in the following manner:\\
            \[Deg(A) = \color{blue}2m(T[A])  + k\langle A,B \rangle \color{black} \leq m(T) \leq \color{brown} 2m(T[B]) + k\langle A,B \rangle \color{black} = Deg(B)
            \]
            \[\color{blue}2(a(T) - \alpha(T)) + k\langle A,B \rangle \leq 2m(T[A])  + k\langle A,B \rangle \color{black} \leq m(T) \leq \] \[ \leq \color{brown} 2m(T[B]) + k\langle A,B \rangle \leq 2(n(T[B])-1) + k\langle A,B \rangle \color{black}.\]
            Therefore,
            \[\color{blue}2(a(T) - \alpha(T)) + k\langle A,B \rangle \color{black}\leq \color{brown}2(n(T[B])-1) + k\langle A,B \rangle \color{black}\]
            \[a(T) - \alpha(T) \leq n(T[B])-1\]
            \[a(T) - \alpha(T) \leq n(T)-a(T)-1\]
            \[a(T) - \alpha(T) \leq \alpha(T) + \mu(T) - a(T) - 1\]
            \[2(a(T) - \alpha(T)) \leq \mu(T) - 1\]
            \[a(T) - \alpha(T) \leq \frac{\mu(T)}{2} - \frac{1}{2}.\]
        \end{proof}
        
        \begin{theorem}
            The inequality $a(T) - \alpha(T) \leq \frac{\mu(T)}{2} - \frac{1}{2}$ represents the tight bound. In other words, there exist a tree for which $a(T) - \alpha(T) = \frac{\mu(T)}{2} - \frac{1}{2}$.
        \end{theorem}

        \begin{proof}
        Let us take a star graph with $6$ vertices. In this case, we have $a(T) = 5$, $\alpha(T) = 5$, and $\mu(T) = 1$. As demonstrated, the equality holds for this tree.
        
        \begin{figure}[H]
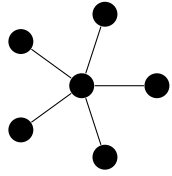

            \centering
            \stargraph{5}{1}
            \caption{A star graph with 6 vertices.}
            \label{fig:star_graph}
        \end{figure}
        
        \end{proof}
        
        Notice that we can expand and characterize the aforementioned graph as follows.
        
        \begin{lemma}
            There exists an infinite family of graphs that satisfy the equality $a(T) - \alpha(T) = \frac{\mu(T)}{2} - \frac{1}{2}$ with $\mu(T) = 1$.
        \end{lemma}

        \begin{proof}
        Let us take a star graph of size $n \geq 2$. It has $a(T) = n(T)-1$, $\alpha(T) = n(T)-1$, and $\mu(T) = 1$. Thus, we can observe that $a(T) - \alpha(T) = \frac{\mu(T)}{2} - \frac{1}{2}$.
        \end{proof}
        
        \begin{lemma}
            There exists an infinite family of graphs that satisfy the equality $a(T) - \alpha(T) = \frac{\mu(T)}{2} - \frac{1}{2}$ with $\mu(T) = 3$.
        \end{lemma}

        \begin{proof}

        For graph with $n$ vertices ($n \geq 6$) having the structure shown in Fig. \ref{fig:mu3}, the degree sequence can be represented as follows:
        \[\underbrace{1,\cdots,1}_\text{$n-3$},
        \underbrace{2,2}_\text{$2$},
        \underbrace{n-3}_\text{$1$}.\]
        \begin{figure}[H]
            \centering
            \begin{tikzpicture}
                \tikzstyle{black}=[fill=black, draw=black, shape=circle]
        		\node [fill=red, draw=black, shape=circle] (0) at (0, 0) {};
        		\node [style=black] (1) at (2, 0) {};
        		\node [style=black] (2) at (3, 0) {};
        		\node [style=black] (3) at (4, 0) {};
        		\node [style=black] (4) at (1, 2) {};
        		\node (100) at (0, -0.5) {Optional vertices};
        		\node [style=black] (5) at (3, -1.75) {};
        		\node [style=black] (6) at (4, -1.75) {};
        		\draw [dotted, red] (0) to (1);
        		\draw (4) to (3);
        		\draw (4) to (2);
        		\draw (4) to (1);
        		\draw [dashed, red](4) to (0);
        		\draw (2) to (5);
        		\draw (3) to (6);
            \end{tikzpicture}
            \caption{An example of a graph with $\mu(T) = 3$ where the tight bound is achieved.}
            \label{fig:mu3}
        \end{figure}
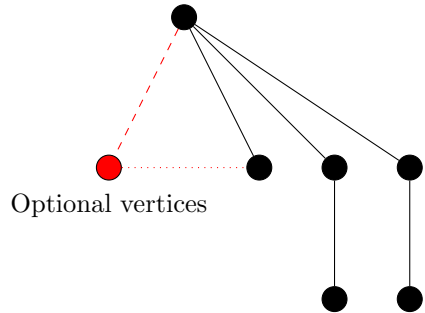
        
        For such graphs, it has $m(T) = n(T)-1$, $a(T) = n(T)-2$, $\alpha(T) = n(T)-3$, and $\mu(T) = 3$. This configuration satisfies the tight bound $a(T) - \alpha(T) = \frac{\mu(T)}{2} - \frac{1}{2}$.
        \end{proof}

        \begin{lemma}
            There exists an infinite family of graphs that satisfy the equality $a(T) - \alpha(T) = \frac{\mu(T)}{2} - \frac{1}{2}$ with $\mu(T) = 5$.
        \end{lemma}
        \begin{proof}
        For any graph with $n$ vertices ($n \geq 10$) having the structure shown in Fig. \ref{fig:mu5}, the degree sequence can be represented as follows:
        \[\underbrace{1,\cdots,1}_\text{$n-5$},
        \underbrace{2,\cdots,2}_\text{$4$},
        \underbrace{n-5}_\text{$1$}.\]        
        \begin{figure}[H]
            \centering
            \begin{tikzpicture}
            
                \tikzstyle{black}=[fill=black, draw=black, shape=circle]
                
        		\node [fill=red, draw=black, shape=circle] (0) at (0, 0) {};
        		\node [style=black] (1) at (1, 2) {};
        		\node [style=black] (2) at (2, 0) {};
        		\node [style=black] (3) at (3, 0) {};
        		\node [style=black] (4) at (4, 0) {};
        		\node [style=black] (5) at (5, 0) {};
        		\node [style=black] (6) at (6, 0) {};
        		\node [style=black] (7) at (6, -2) {};
        		\node [style=black] (8) at (5, -2) {};
        		\node [style=black] (9) at (4, -2) {};
        		\node [style=black] (10) at (3, -2) {};
        		\node (100) at (0, -0.5) {Optional vertices};

        		\draw [dotted, red] (0) to (2);
        		\draw [dashed, red](0) to (1);
        		\draw (1) to (6);
        		\draw (1) to (5);
        		\draw (1) to (4);
        		\draw (1) to (3);
        		\draw (1) to (2);
        		\draw (6) to (7);
        		\draw (5) to (8);
        		\draw (4) to (9);
        		\draw (3) to (10);
            \end{tikzpicture}
            \caption{An example of a graph with $\mu(T) = 5$ where the tight bound is achieved.}
            \label{fig:mu5}
        \end{figure}
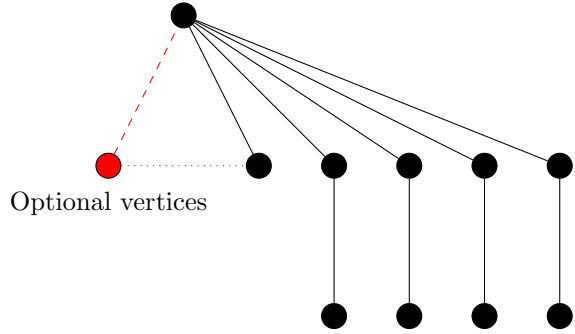

        For such graphs, it has $m(T) = n(T)-1$, $a(T) = n(T)-3$, $\alpha(T) = n(T)-5$, and $\mu(T) = 5$. Thus, this configuration satisfies the tight bound $a(T) - \alpha(T) = \frac{\mu(T)}{2} - \frac{1}{2}$.
        \end{proof}
        \begin{theorem}
            There exists an infinite family of graphs that satisfy the equality $a(T) - \alpha(T) = \frac{\mu(T)}{2} - \frac{1}{2}$ with $\mu(T) = 2q+1$, where $q$ is a positive integer.
        \end{theorem}
        \begin{proof}
        We can deduce the following from the previous three lemmas:
        \begin{itemize}
            \item The number of vertices is $n(T) = 4q+2$,
            \item The independence number is $\alpha(T) = 2q+1$,
            \item There is a root of degree $2q+1$,
            \item There are $2q$ vertices of degree $2$,
            \item There are $2q+1$ vertices with degree $1$.
        \end{itemize}

        $T$ is a tree, so $m(T) = 4q+1$. Now, let us consider the sum of the degrees of the vertices until we reach $m(T)$. There are $2q+1$ vertices that are leaves, and we select $q$ vertices from the $2q$ vertices with degree $2$. Therefore, the sum becomes:
        \[(2q+1)\cdot 1 + q\cdot2 = 4q+1\]
        and $a(T) = 3q+1$. Substituting these values into the equation, we get,
        \[\underbrace{3q+1}_\text{a(T)} - \underbrace{2q+1}_\text{$\alpha(T)$} = \underbrace{\frac{2q+1}{2}}_\text{$\frac{\mu(T)}{2}$} - \frac{1}{2}.\]
        \end{proof}
        \begin{observation}
                $\mu(T)$ can only be an odd number in this case. This is because the difference on the left side of the equality, $a(T) - \alpha(T)$, is an integer, and the right side, $\frac{\mu(T)}{2} - \frac{1}{2}$, must also be an integer. Since $\frac{\mu(T)}{2}- \frac{1}{2}$ is not an integer when $\mu(T)$ is even, the only possibility is for $\mu(T)$ to be an odd number.
        \end{observation}

        \begin{lemma}
            There exist bipartite graphs that do not satisfy the inequality $a(G) - \alpha(G) \leq \frac{\mu(G)}{2} - \frac{1}{2}$.
        \end{lemma}
        
        \begin{proof}
        Let us take a bipartite graph $G$ with $V(G) = A \cup B \cup C$, where $|A|=16, |B|=8, |C|=8$. $A$ (red vertices in Figure \ref{fig:Bipartite_graph_not_satisfying_the_tree_inequality}) and $B$ (blue vertices in Figure \ref{fig:Bipartite_graph_not_satisfying_the_tree_inequality}) are independent sets, and $C$ (black vertices in Figure \ref{fig:Bipartite_graph_not_satisfying_the_tree_inequality}) is a complete bipartite graph $K_{4,4}$. Half of the vertices from $A$ are connected to $B$ by a matching, while another half of $A$ are connected to $C$ by a matching. Additionally, each vertex from $B$ is connected to one vertex from $C$.
            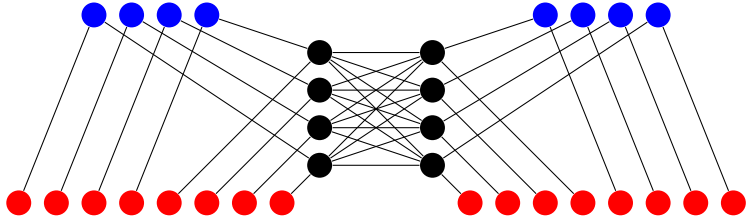
\begin{figure}[H]
                \centering
                \begin{tikzpicture}[scale=0.5]
            		\node [fill=black,circle] (20) at (11, 1) {};
            		\node [fill=black,circle] (21) at (8, 1) {};
            		\node [fill=black,circle] (22) at (8, 2) {};
            		\node [fill=black,circle] (23) at (8, 3) {};
            		\node [fill=black,circle] (24) at (8, 4) {};
            		\node [fill=black,circle] (25) at (11, 4) {};
            		\node [fill=black,circle] (26) at (11, 3) {};
            		\node [fill=black,circle] (27) at (11, 2) {};
            		\node [fill=red,circle] (28) at (12, 0) {};
            		\node [fill=red,circle] (29) at (13, 0) {};
            		\node [fill=red,circle] (30) at (14, 0) {};
            		\node [fill=red,circle] (31) at (15, 0) {};
            		\node [fill=red,circle] (32) at (16, 0) {};
            		\node [fill=red,circle] (33) at (17, 0) {};
            		\node [fill=red,circle] (34) at (18, 0) {};
            		\node [fill=red,circle] (35) at (19, 0) {};
            		\node [fill=red,circle] (36) at (7, 0) {};
            		\node [fill=red,circle] (37) at (6, 0) {};
            		\node [fill=red,circle] (38) at (5, 0) {};
            		\node [fill=red,circle] (39) at (4, 0) {};
            		\node [fill=red,circle] (40) at (3, 0) {};
            		\node [fill=red,circle] (41) at (2, 0) {};
            		\node [fill=red,circle] (42) at (1, 0) {};
            		\node [fill=red,circle] (43) at (0, 0) {};
            		\node [fill=blue,circle] (44) at (17, 5) {};
            		\node [fill=blue,circle] (45) at (16, 5) {};
            		\node [fill=blue,circle] (46) at (15, 5) {};
            		\node [fill=blue,circle] (47) at (14, 5) {};
            		\node [fill=blue,circle] (48) at (2, 5) {};
            		\node [fill=blue,circle] (49) at (3, 5) {};
            		\node [fill=blue,circle] (50) at (4, 5) {};
            		\node [fill=blue,circle] (51) at (5, 5) {};
            		\draw (21) to (20);
            		\draw (21) to (27);
            		\draw (21) to (26);
            		\draw (21) to (25);
            		\draw (22) to (20);
            		\draw (22) to (27);
            		\draw (22) to (26);
            		\draw (22) to (25);
            		\draw (23) to (20);
            		\draw (23) to (27);
            		\draw (23) to (26);
            		\draw (23) to (25);
            		\draw (24) to (20);
            		\draw (24) to (27);
            		\draw (24) to (26);
            		\draw (24) to (25);
            		\draw (36) to (21);
            		\draw (37) to (22);
            		\draw (38) to (23);
            		\draw (39) to (24);
            		\draw (43) to (48);
            		\draw (42) to (49);
            		\draw (41) to (50);
            		\draw (40) to (51);
            		\draw (51) to (24);
            		\draw (50) to (23);
            		\draw (49) to (22);
            		\draw (48) to (21);
            		\draw (47) to (25);
            		\draw (46) to (26);
            		\draw (45) to (27);
            		\draw (44) to (20);
            		\draw (44) to (35);
            		\draw (45) to (34);
            		\draw (46) to (33);
            		\draw (47) to (32);
            		\draw (25) to (31);
            		\draw (26) to (30);
            		\draw (27) to (29);
            		\draw (20) to (28);
                \end{tikzpicture}
                \caption{A bipartite graph not satisfying the inequality for trees.}
                \label{fig:Bipartite_graph_not_satisfying_the_tree_inequality}
            \end{figure}
             In this given example, there is a bipartite graph with $n = 32$ vertices and $a(G)=25, \alpha(G)=16, \mu(G)=16$ and $25 - 16 > \frac{16-1}{2}.$
        \end{proof}        
    
        Given that every tree is a K\"{o}nig--Egerv\'{a}ry graph, Lemma \ref{KE_alpha>=mu} can be applied that for trees $\mu(T) \leq \alpha(T)$. Now, we can express the tight bound for trees solely in terms of the annihilation number $a(T)$ and the independence number $\alpha(T)$ as follows.

        \begin{corollary}\label{corollary_tree}
             If T is a tree, then $a(T) \leq \frac{3}{2}\alpha(T) - \frac{1}{2}$.
        \end{corollary}
        \begin{proof}
        \[a(T) - \alpha(T) \leq \frac{\mu(T)}{2} - \frac{1}{2} \leq  \frac{\alpha(T)}{2} - \frac{1}{2}\]
            \[a(T) \leq \frac{3}{2}\alpha(T) - \frac{1}{2}.\] 
        \end{proof}


    \section{The tight upper bound for bipartite graphs} \label{bipartite_section}
        To establish the tight upper bound for bipartite graphs, we make use of the following well-known lemma.
        \begin{lemma}\label{lemma:edges=n^2/4}
            The maximum number of edges in a bipartite graph $G$ is $\frac{n(G)^2}{4}$.
        \end{lemma}

        \begin{theorem}\label{theorem_bipartite}
            For every bipartite graph $G$, $a(G) - \alpha(G) \leq 2 + \mu(G) - 2\sqrt{1 + \mu(G)}$.
        \end{theorem}
        \begin{proof}

            Consider a bipartite graph $G$ with an annihilation decomposition $\langle A, B \rangle$. 
            \[Deg(A) = \color{blue}2m(G[A])  + k\langle A, B \rangle \color{black} \leq m(G) \leq \color{brown} 2m(G[B]) + k\langle A, B \rangle \color{black} = Deg(B)\text{\;\;(Lemma \ref{lemma:ma<mb})}\]
            \[\color{blue}2(a(G) - \alpha(G)) + k\langle A, B \rangle \leq 2m(G[A]) + k\langle A, B \rangle \color{black}\text{\;\;(Lemma \ref{lemma:in_bipartite_a-alpha<m(A)})}\]
            \[\color{brown}2m(G[B]) + k\langle A, B \rangle \leq 2\left(\frac{(n(G)-a(G))^2}{4} \right) + k\langle A, B \rangle\color{black} \text{\;\;(Lemma \ref{lemma:edges=n^2/4})}\]
            Combining the above inequalities, we obtain
            \[\color{blue}2(a(G) - \alpha(G)) + k\langle A, B \rangle \leq 2m(G[A])  + k\langle A, B \rangle \color{black} \leq m(G) \leq\] \[\leq\color{brown} 2m(G[B]) + k\langle A, B \rangle \leq 2\left(\frac{(n(G)-a(G))^2}{4} \right) + k\langle A, B \rangle. \color{black}\]
            Hence,
            \[\color{blue}2(a(G) - \alpha(G)) + k\langle A, B \rangle \color{black}\leq \color{brown}2\left(\frac{(n(G)-a(G))^2}{4}\right) + k\langle A, B \rangle\color{black}\]
            \[a(G) - \alpha(G) \leq \frac{(n(G)-a(G))^2}{4}\]
            \[4(a(G) - \alpha(G)) \leq (\alpha(G) - a(G) + \mu(G))^2\]
            \[4(a(G) - \alpha(G)) \leq (\alpha(G) - a(G))^2 - 2\mu(G)(a(G) - \alpha(G)) + \mu(G)^2\]
            \[(a(G) - \alpha(G))^2 - (2\mu(G) + 4)(a(G) - \alpha(G)) + \mu(G)^2 \geq 0\]
            The zeros of the equation are:
            \[(a(G) - \alpha(G))_{1,2} = 2 + \mu(G) \pm 2\sqrt{1+\mu(G)}.\]
            In the last  equality, we look for values that are bigger or equal to zero, so we have only two cases,\\
    
            \textbf{Case 1:} $a(G) - \alpha(G) \geq 2 + \mu(G) + 2\sqrt{1+\mu(G)}$\\
            By Lemma \ref{a_minum_alpha_leq_mu}, For K\"{o}nig--Egerv\'{a}ry graphs $a(G)-\alpha(G) \leq \mu(G)$. Hence, this case is impossible.\\
            
            \textbf{Case 2:} $a(G) - \alpha(G) \leq 2 + \mu(G) - 2\sqrt{1+\mu(G)}$\\
            By Lemma \ref{a_minum_alpha_leq_mu}, $a(G) - \alpha(G) \leq \mu(G)$, and easy to see that $2 - 2\sqrt{1+\mu(G)}$ is a negative number. So this case holds for every $\mu(G) > 0$. 
        \end{proof}
        
        \begin{theorem}
            The inequality $a(G) - \alpha(G) \leq 2 + \mu(G) - 2\sqrt{1+\mu(G)}$ is the tight bound for bipartite graphs.
        \end{theorem}
        
        \begin{proof}
            
        As an illustration, Figure \ref{fig:example_of_tight_bipartite} shows the case of a bipartite graph with $n(G) = 16$, where $a(G) = 12$, $\alpha(G) = 8$, and $\mu(G) = 8$. Notably, this graph satisfies the equality corresponding to the tight bound: $12 - 8 = 2 + 8 - 2\sqrt{1+8}$ and not satisfies the tight bound for trees: $12 - 8 > \frac{8}{2} - \frac{1}{2}$.
            
            \begin{figure}[H]
                \centering
                
                \begin{tikzpicture}[scale=0.5]
                    \tikzstyle{black}=[fill=black, draw=black, shape=circle]
                    \tikzstyle{red}=[fill=red, draw=black, shape=circle]
                    \tikzstyle{blue}=[fill=blue, draw=black, shape=circle]
                
            		\node [style=red] (0) at (0, 0) {};
            		\node [style=red] (1) at (1, 0) {};
            		\node [style=red] (2) at (2, 0) {};
            		\node [style=red] (3) at (3, 0) {};
            		\node [style=black] (4) at (4, 1) {};
            		\node [style=black] (5) at (5, 1) {};
            		\node [style=black] (6) at (4, 2) {};
            		\node [style=black] (7) at (5, 2) {};
            		\node [style=red] (8) at (6, 0) {};
            		\node [style=red] (9) at (7, 0) {};
            		\node [style=red] (10) at (8, 0) {};
            		\node [style=red] (11) at (9, 0) {};
            		\node [style=blue] (12) at (7, 3) {};
            		\node [style=blue] (13) at (8, 3) {};
            		\node [style=blue] (14) at (2, 3) {};
            		\node [style=blue] (15) at (1, 3) {};
            
            		\draw (6) to (7);
            		\draw (6) to (5);
            		\draw (4) to (7);
            		\draw (4) to (5);
            		\draw (3) to (4);
            		\draw (2) to (6);
            		\draw (1) to (14);
            		\draw (0) to (15);
            		\draw (14) to (6);
            		\draw (15) to (4);
            		\draw (12) to (7);
            		\draw (13) to (5);
            		\draw (13) to (11);
            		\draw (12) to (10);
            		\draw (7) to (9);
            		\draw (5) to (8);
                \end{tikzpicture}

                \caption{A bipartite graph that satisfies the equality.}
                \label{fig:example_of_tight_bipartite}
            \end{figure}
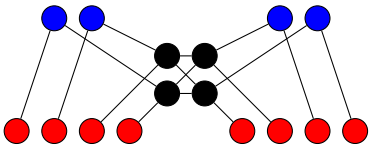

        \end{proof}

        \begin{theorem}
            There exists an infinite family of graphs where the difference between the annihilation number $a(G)$ and the independence number $\alpha(G)$ is equal to $2 + \mu(G) - 2\sqrt{1+\mu(G)}$.
        \end{theorem}
        \begin{proof}
        Let $p\geq 3$ be an  integer. By observing Figure \ref{fig:example_of_tight_bipartite}, it becomes apparent that the vertex set can be divided into three distinct sets:
        \begin{itemize}
            \item A maximum independent set, denoted as $A$, which consists of $p^2-1$ vertices. In this set, we have $(p-1)^2$ vertices connected to $B$ in a one-to-one manner and $2p-2$ vertices connected to $C$ in a one-to-one manner. All the vertices in set $A$ have a degree of 1, meaning that each vertex is connected to exactly one other vertex in the graph.
            \item Similarly, independent set denoted as $B$, which contains $(p-1)^2$ vertices. In this set, all the vertices in $B$ are connected to vertices in $A$ in a one-to-one manner. Additionally, $2p-2$ vertices from set $B$ are also connected to all the vertices in $C$ in a one-to-one manner. As a result, there are $p^2-4p+3$ vertices in $B$ with a degree of 1. Moreover, there are $2p-2$ vertices in $B$ with a degree of 2, as they are connected to both vertices in $A$ and vertices in $C$.
            \item Set $C$, which is a complete bipartite graph $K_{p-1,p-1}$ consisting of $2p-2$ vertices. In this set, every vertex in $C$ is connected to one vertex in $A$ and one vertex in $B$, forming one-to-one connections. Consequently, all vertices in set $C$ have a degree equal to $p+1$.
        \end{itemize}
        
        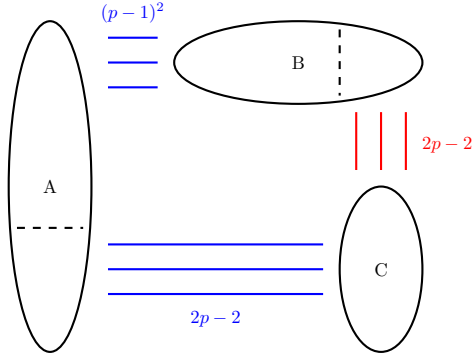
\begin{figure}[H]
            \centering

            \begin{tikzpicture}[thick,scale=1.1, every node/.style={scale=0.7}]
                \draw[] (0,0) ellipse (0.5cm and 2cm);  
                \draw[] (3,1.5) ellipse (1.5cm and 0.5cm);
                \draw[] (4,-1) ellipse (0.5cm and 1cm);
                \node at (0,0) {A};
                \draw[thick] (-0.4,-0.5) -- (0.4,-0.5)[dashed] ;
                \node at (3,1.5) {B};
                \draw[thick] (3.5,1.9) -- (3.5,1.1)[dashed] ;
                \node at (4,-1) {C};
                \draw[thick,blue] (0.7,-0.7) -- (3.3,-0.7) ;
                \draw[thick,blue] (0.7,-1) -- (3.3,-1) ;
                \draw[thick,blue] (0.7,-1.3) -- (3.3,-1.3) ;
                \node[blue] at (2,-1.6) {$2p-2$};
                
                \draw[thick,blue] (0.7,1.8) -- (1.3,1.8) ;
                \draw[thick,blue] (0.7,1.5) -- (1.3,1.5) ;
                \draw[thick,blue] (0.7,1.2) -- (1.3,1.2) ;
                \node[blue] at (1,2.1) {$(p-1)^2$};
                
                \draw[thick,red] (3.7,0.9) -- (3.7,0.2) ;
                \draw[thick,red] (4,0.9) -- (4,0.2) ;
                \draw[thick,red] (4.3,0.9) -- (4.3,0.2) ;
                \node[red] at (4.8,0.5) {$2p-2$};
            \end{tikzpicture}
        
            \caption{The partition of bipartite graphs where the tight bound is achieved.}
            \label{fig:decomposition_of_graph_family}
                        
        \end{figure}

        The degree sequence is $\underbrace{1,\cdots,1}_\text{$p^2-1$},
        \underbrace{1,\cdots,1}_\text{$p^2-4p+3$},
        \underbrace{2,\cdots,2}_\text{$2p-2$},
        \underbrace{p+1,\cdots,p+1}_\text{$2p-2$}.$\\

        To ensure that we avoid negative values for $p^2-4p+3$, it is suggested that $p \geq 3$.\\

        In this partition, we observe that $\alpha(G) = p^2-1 = |A|$ and $\mu(G) = p^2-1 = |A|$. Now, $m(G)$ can be calculated as follows:
        
        \begin{center}
            \[m(G)=\frac{1\cdot(p^2-1) + 1\cdot(p^2-4p+3) +2\cdot (2p-2)+(2p-2)\cdot (p+1)}{2}\]
            \[=2p^2-2.\]      
        \end{center}
  
        Let us calculate $a(G)$. We consider all the vertices in sets $A$ and $B$. These vertices have degrees of either $1$ or $2$.
        \begin{center}
             $\underbrace{1,\cdots,1, 1,\cdots,1, 2,\cdots,2},
            p+1,\cdots,p+1.$            
        \end{center}
        Now,
         \[1\cdot(p^2-1) + 1\cdot(p^2-4p+3) +2\cdot (2p-2)=2p^2-2=m(G).\]
        
        Hence, the maximum value that $a(G)$ can be is given by, \[a(G)= (p^2-1)+(p^2-4p+3)+(2p-2)=2p^2-2p.\]
        And, the difference between $a(G)$ and $\alpha(G)$ is, \[a(G)-\alpha(G) = 2p^2-2p - (p^2-1) = p^2-2p+1=(p-1)^2.\]

        Another way to calculate this difference is by setting $\mu(G)$ to be $p^2-1$ on the right side of the bipartite inequality bound. We have \[2+\mu(G)-2\sqrt{1+\mu(G)} = 2+(p^2-1)-2\sqrt{1+(p^2-1)} = p^2-2p+1=(p-1)^2.\]
        As we can see, we obtained the same expression for the difference between $a(G)$ and $\alpha(G)$.\\
        \end{proof}

        \begin{observation}
            If $p = 3$, then the graph $G$ is connected. On the other hand, if $p \geq 4$, then the graph $G$ is disconnected.
        \end{observation}
    
        \begin{lemma}
           There exist a K\"{o}nig--Egerv\'{a}ry graph that do not satisfy the inequality $a(G) - \alpha(G) \leq 2 + \mu(G) - 2\sqrt{1+\mu(G)}$ .
        \end{lemma}
        \begin{proof}
            
        In the graph shown in Figure \ref{fig:KE_not_bipartite}, there are 12 vertices. The blue vertices represent the independent set with an independence number of 6. The red vertices represent the matching with a matching number also equal to 6. Therefore, $G$ is a K\"{o}nig--Egerv\'{a}ry graph. The degree sequence of this graph is:
        
            \[1, 2, 2, 2, 3, 3, 3, 3, 3, 7, 7, 8\]
            and the annihilation number is $9$. In this example, we have $a(G) - \alpha(G) = 9 - 6 = 3$, which is greater than the value obtained from the inequality $2 + \mu(G) - 2\sqrt{1 + \mu(G)} = 2+6-2\sqrt{1+6} = 2.7 .$
            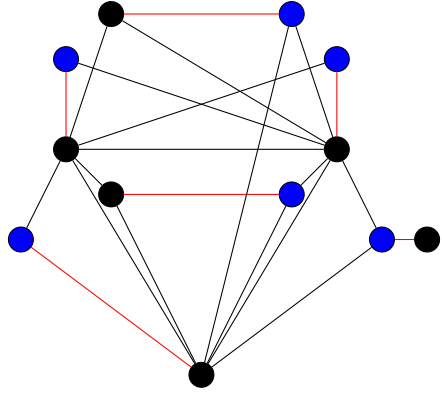
\begin{figure}[H]
                \centering
                \begin{tikzpicture}[scale=0.6]
                \tikzstyle{black}=[fill=black, draw=black, shape=circle]
                \tikzstyle{blue}=[fill=blue, draw=black, shape=circle]
            		\node [style=black] (0) at (0, 0) {};
            		\node [style=blue] (1) at (-4, 3) {};
            		\node [style=blue] (2) at (4, 3) {};
            		\node [style=black] (3) at (5, 3) {};
            		\node [style=black] (4) at (-2, 4) {};
            		\node [style=blue] (5) at (2, 4) {};
            		\node [style=black] (6) at (-3, 5) {};
            		\node [style=black] (7) at (3, 5) {};
            		\node [style=blue] (8) at (-3, 7) {};
            		\node [style=blue] (9) at (3, 7) {};
            		\node [style=black] (10) at (-2, 8) {};
            		\node [style=blue] (11) at (2, 8) {};
            		\draw [red] (0) to (1);
            		\draw (0) to (6);
            		\draw (0) to (4);
            		\draw (0) to (5);
            		\draw (0) to (7);
            		\draw (0) to (2);
            		\draw [red](2) to (3);
            		\draw (2) to (7);
            		\draw (7) to (5);
            		\draw (4) to (6);
            		\draw (6) to (1);
            		\draw [red](4) to (5);
            		\draw (6) to (7);
            		\draw [red](7) to (9);
            		\draw (9) to (6);
            		\draw [red](6) to (8);
            		\draw (8) to (7);
            		\draw [red](10) to (11);
            		\draw (11) to (7);
            		\draw (10) to (6);
            		\draw (11) to (0);
            		\draw (10) to (7);
                \end{tikzpicture}

                \caption{A K\"{o}nig--Egerv\'{a}ry graph not satisfying the bipartite tight bound.}
                \label{fig:KE_not_bipartite}
            \end{figure}
        
        \end{proof}
        By utilizing Lemma \ref{KE_alpha>=mu}, we can rephrase the bipartite inequality to solely involve $a$ and $\alpha$ as follows.
        \begin{corollary}\label{corollary_bipartite}
            If $G$ is a bipartite graph, then $a(G) \leq 2 + 2\alpha(G) - 2\sqrt{1 + \alpha(G)}.$
        \end{corollary} 
        \begin{proof}
             \[a(G) - \alpha(G) \leq 2 + \mu(G) - 2\sqrt{1 + \mu(G)} \leq 2 + \alpha(G) - 2\sqrt{1 + \alpha(G)}\]  
             \[a(G) \leq 2 + 2\alpha(G) - 2\sqrt{1 + \alpha(G)}.\] 
        \end{proof}


    \section{The tight upper bound for K\"{o}nig--Egerv\'{a}ry graphs}\label{KE_section}
    To establish the tight upper bound for K\"{o}nig--Egerv\'{a}ry graphs, we present the following lemmas.
        \begin{lemma}
            If a graph $G$ is connected, the inequality $a(G) \leq n(G)-1$ holds only for star graphs.
        \end{lemma}
        \begin{proof}
        Consider a graph $G$ with an annihilation number $a(G)$ equal to $n(G)-1$. The sum of the degrees of the first $n(G)-1$ vertices is equal to the degree of the last vertex in the degree sequence. However, the maximum degree of the last vertex can only be $n(G)-1$, which is the case for star graphs.
        \end{proof}
        \begin{corollary}\label{corollary_KE_a<n-1}
            For every connected graph that is not a star graph, the annihilation number $a(G)$ is bounded by $n(G)-2$.
        \end{corollary}
        
        \begin{theorem}\label{KE_tight_bound}
            For every K\"{o}nig--Egerv\'{a}ry graph, the inequality $a(G) - \alpha(G) \leq \mu(G) - 2$ represents the tight bound. 
        \end{theorem}
        \begin{proof}
        Consider a K\"{o}nig--Egerv\'{a}ry graph $G$ that is not a bipartite graph, with the following structure: Start with a cycle $C_3$ containing vertices $v_1$, $v_2$, and $v_3$. Vertex $v_1$ is connected to an independent set of size $k$, vertex $v_2$ is connected to another independent set of size $k$, and vertex $v_3$ is connected to a single additional vertex. This graph has a total of $2k+4$ vertices, $\alpha(G) = 2k+1$, and $\mu(G) = 3$.
        
        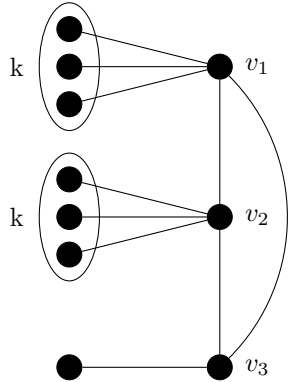
\begin{figure}[H]
            \centering
            \begin{tikzpicture}[scale=0.5]
                \tikzstyle{black}=[fill=black, draw=black, shape=circle]

        		\node [style=black] (0) at (0, -1) {};
        		\node [style=black] (1) at (0, 2) {};
        		\node [style=black] (2) at (0, 3) {};
        		\node [style=black] (3) at (0, 4) {};
        		\node [style=black] (4) at (0, 6) {};
        		\node [style=black] (5) at (0, 7) {};
        		\node [style=black] (6) at (0, 8) {};
        		\node [style=black] (7) at (4, 7) {};
        		\node [style=black] (8) at (4, 3) {};
        		\node [style=black] (9) at (4, -1) {};
        		\draw (9) to (8);
        		\draw (8) to (7);
        		\draw [bend left=45, looseness=1] (7) to (9);
        		\draw (0) to (9);
        		\draw (1) to (8);
        		\draw (2) to (8);
        		\draw (3) to (8);
        		\draw (4) to (7);
        		\draw (5) to (7);
        		\draw (6) to (7);
        		
        		\draw[] (0,3) ellipse (0.8cm and 1.7cm);
        		\draw[] (0,7) ellipse (0.8cm and 1.7cm);
        		
                \node at (-1.4,3) {k};
                \node at (-1.4,7) {k};
                \node at (5, 7) {$v_1$};
                \node at (5, 3) {$v_2$};
                \node at (5, -1) {$v_3$};
            \end{tikzpicture}
        
            \caption{Family of infinity graphs that satisfy the tight bound with $\mu=3$}
            \label{fig:KE_general_example}
        \end{figure}
        In this structure, the degree sequence is,
        \begin{center}
            $\underbrace{1,\cdots,1}_\text{$2k+1$},
            \underbrace{3}_\text{$1$},
            \underbrace{k+2}_\text{$1$},
            \underbrace{k+2}_\text{$1$}$\\
        \end{center}
        
        The number of edges in each graph is equal to $2k+4$. The annihilation number $a(G)$ includes all the vertices with a degree of 1 or 3, so $a(G) = 2k+2$. Hence, we have $a(G) = n(G) - 2$, which can be expressed differently as $a(G) - \alpha(G) = \mu(G) - 2$.\\

        \end{proof}
        
        For example, Figure \ref{fig:example_KE_graph_a=n_minus2} illustrates the configuration from Figure \ref{fig:KE_general_example} with $k=2$.
        
        \begin{figure}[H]
            \centering
            \begin{tikzpicture}[scale=0.7]
        	\tikzstyle{black}=[fill=black, draw=black, shape=circle]
        		\node [style=black] (0) at (0, -0.5) {};
        		\node [style=black] (1) at (0, 1) {};
        		\node [style=black] (2) at (0, 2) {};
        		\node [style=black] (3) at (0, 3) {};
        		\node [style=black] (4) at (0, 4) {};
        		\node [style=black] (5) at (2, 3.5) {};
        		\node [style=black] (6) at (2, 1.5) {};
        		\node [style=black] (7) at (2, -0.5) {};
        		\draw (0) to (7);
        		\draw (1) to (6);
        		\draw (2) to (6);
        		\draw (3) to (5);
        		\draw (4) to (5);
        		\draw (5) to (6);
        		\draw (6) to (7);
        		\draw [in=-45, out=45, looseness=1] (7) to (5);
            \end{tikzpicture}
            \caption{Example of a K\"{o}nig--Egerv\'{a}ry graph with $n=8$ and $a=6$.}
            \label{fig:example_KE_graph_a=n_minus2}
        \end{figure}
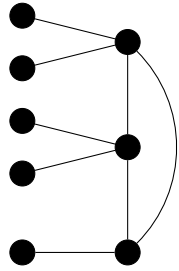

        Based on Theorem \ref{KE_tight_bound} and Lemma \ref{KE_alpha>=mu}, we can rephrase the König-Egerváry inequality to involve only $a(G)$ and $\alpha(G)$ as follows:
        \begin{corollary}\label{corollary_KE}
            For K\"{o}nig--Egerv\'{a}ry graphs that are not bipartite, we have $a(G) \leq 2\alpha(G) - 2$.
        \end{corollary}
        
        \begin{proof}
        \[a(G) - \alpha(G) \leq \mu(G) - 2 \leq \alpha(G) - 2,\] 
        \[a(G) \leq 2\alpha(G) - 2.\]
        \end{proof}


    \section{Conclusions}
        This paper investigates three tight bounds on the difference between the annihilation number $a(G)$ and the independence number $\alpha(G)$. These bounds are as follows:
        \begin{itemize}
            \item For every tree $T$, we have $a(T) - \alpha(T) \leq \frac{\mu(T)}{2} - \frac{1}{2}$.
            \item For every bipartite graph $G$, we have $a(G) - \alpha(G) \leq 2 + \mu(G) - 2\sqrt{1 + \mu(G)}$.
            \item For every K\"{o}nig--Egerv\'{a}ry graph $G$, we have $a(G) - \alpha(G) \leq \mu(G) - 2$.
        \end{itemize}
        Those three inequalities can be rewritten by Corollaries \ref{corollary_tree},  \ref{corollary_bipartite} and \ref{corollary_KE} by the following.
        \begin{itemize}
            \item For every tree $T$, $a(T) \leq \frac{3}{2}\alpha(T) - \frac{1}{2}$.
            \item For every bipartite graph $G$, $a(G) \leq 2 + 2\alpha(G) - 2\sqrt{1 + \alpha(G)}$.
            \item For every K\"{o}nig--Egerv\'{a}ry graph $G$, $a(G) \leq 2\alpha(G) - 2$.
        \end{itemize}
        \begin{problem}
            Find the tight bound on the difference between $a(G)$ and $\alpha(G)$ for forests.
        \end{problem}

        \begin{conjecture}
            There exists an infinite family of connected bipartite graphs that illustrate the tight bound in Theorem \ref{theorem_bipartite}.
        \end{conjecture}
            
        \begin{conjecture}
             The minimum number of vertices for a bipartite graph satisfying the equality $a(G) - \alpha(G) = 2 + \mu(G) - 2\sqrt{1+\mu(G)}$ is 16.
             
        \end{conjecture}
        \begin{problem}
            Mantel's theorem states that for triangle-free graphs, the number of edges $m$ is at most $\frac{n(G)^2}{4}$, which is the same bound as for bipartite graphs. Find the tight upper bound for the difference between $a(G)$ and $\alpha(G)$ for triangle-free graphs.
        \end{problem}
        
        \begin{problem}
            Find the tight upper bound for the difference between $a(G)$ and $\alpha(G)$ for general graphs.
        \end{problem}

\end{document}